\RequirePackage{fix-cm}
\documentclass{svjour3}                     
\smartqed  
\usepackage{mathptmx}      
\usepackage{amsmath}
\usepackage{amsfonts}
\usepackage{amssymb}
\usepackage{mathrsfs}
\usepackage{booktabs}
\usepackage{microtype}
\usepackage{bm}
\usepackage{xcolor}
\usepackage{enumitem}
\usepackage[shortcuts]{extdash}
\usepackage[hyphens]{url}
\usepackage[colorlinks=true,allcolors=black]{hyperref}

\usepackage[labelfont=bf]{caption}
\usepackage{tikz}
\usetikzlibrary{calc}

\newcommand{\N}{\mathbb{N}}
\newcommand{\R}{\mathbb{R}}

\makeatletter
 \renewcommand*\env@matrix[1][*\c@MaxMatrixCols l]{%
   \hskip -\arraycolsep
   \let\@ifnextchar\new@ifnextchar
   \array{#1}}
\makeatother

\DeclareMathOperator{\Li}{\mathrm{Li}}

\hypersetup{bookmarksdepth=3}

\journalname{}
\begin{document}
\allowdisplaybreaks

\title{%
  Pan-Xu conjecture and reduction formulas for polylogarithms}

\author{Marian Gen\v cev}

\institute{%
\raise 2pt \hbox{\scriptsize\llap{*\,}}%
  Corresponding author: M. Gen\v cev
   (marian.gencev@vsb.cz)\at
   Faculty of Economics, VSB -- Technical University of Ostrava,\\
   17. listopadu 2172/15, 708\,00 Ostrava,
   Czech Republic}

\date{Received: date / Accepted: date}

\maketitle

\begin{abstract}
The objective of the paper is the study of Mneimneh-like sums with a parametric variant of the multiple harmonic-star values. We generalize and resolve the Pan-Xu conjecture on generalized Mneimneh-like sums and present their transformation. As an application, we deduce new reduction formulas for specific multiple polylogarithms enabling lowering their depth, and provide additional findings on arithmetic means of multiple harmonic-star values, resulting in new representations of arbitrary multiple zeta-star values.

\end{abstract}

\DeclareFontSeriesDefault[rm]{bf}{b}

\section{Introduction}

\rlap{\smash{\hskip 2.2cm \raise 10.2cm \hbox{\scriptsize *}}}
In 2023, Mneimneh~\cite{mneimneh} published a paper devoted to applications of specific binomial\discretionary{-}{-}{-}harmonic sums, and proved the simply-looking identity
\begin{equation}
\label{eq:mneimneh}
\sum_{k=1}^n\binom nk\cdot p^k\cdot (1-p)^{n-k}\cdot H_k
	=\sum_{k=1}^n\frac{1-(1-p)^k}k,
\end{equation}
with $0\le p\le 1$ and $H_k:=1+1/2+\cdots+1/k$ denoting the harmonic number. Short after Mneimneh's paper \cite{mneimneh} appeared, Campbell~\cite{campbell} published two additional proofs of~\eqref{eq:mneimneh}; first, by employing an integral representation of~$H_k$, and second, by a~modification of the Wilf\discretionary{-}{-}{-}Zeilberger method. Even if Campbell mentioned that further generalizations of~\eqref{eq:mneimneh} are possible, especially with higher powers of harmonic numbers $H_k^m$, only a few generalizations of \eqref{eq:mneimneh} has been investigated recently. One of the recent modifications of \eqref{eq:mneimneh} was examined by Komatsu and Wang~\cite{komatsu} who focused on the case with hyperharmonic numbers.

Independent of the analyses in~\cite{campbell,komatsu,mneimneh}, we proposed in~\cite{gencev.dm} another generalization of Mneimneh's binomial sums in the form
\begin{equation}
\label{eq:mneimneh.ap.def}
M_n^{(s)}(a,p)
	:=\sum_{k=1}^n\binom nk\cdot p^k\cdot (1-p)^{n-k}\cdot H_k^{(s)}(a),
\end{equation}
where $s\in\N$, $a,p\in\R$, and $H_k^{(s)}(a):=\sum_{j=1}^ka^j/j^s$, $a\in\R$, denotes the generalized harmonic number of order $s$. As a result, we deduced the transformation identity, see \cite[Theorem~2.1 and Eq.~(6)]{gencev.dm},
\begin{equation}
\label{eq:gencev}
M_n^{(s)}(a,p)
	=\sum_{n\ge n_1\ge\cdots\ge n_s\ge 1}
		\frac{(1-p)^{n_1}}{n_1\cdots n_s}\cdot
		\left(\left(1+\frac{ap}{1-p}\right)^{n_s}-1\right)
\end{equation}
with the right-hand side in the form of partial sums of specific nested series. The corresponding infinite series are called {\itshape multiple polylogarithms} and defined by
\begin{equation}
\label{eq:polylog.def}
\Li^\star_{s_1,\dots ,s_d}(x_1,\dots,x_d)
	:=\sum_{n_1\ge\cdots\ge n_d\ge 1}
		\prod_{i=1}^d\frac{x_i^{n_i}}{n_i^{s_i}},
\qquad
s_i\in\N,
\quad
x_i\in\R.
\end{equation}

\begin{remark}
We note that the one-dimensional polylogarithmic function, corresponding to $d=1$ in \eqref{eq:polylog.def}, is called the polylogarithm function. In this simple case, the star symbol is usually dropped, i.e.~$\Li_s(x)=\Li^\star_s(x)=\sum_{n=1}^\infty x^n/n^s$. Furthermore, throughout this paper, we often employ the polylogarithmic values with the arguments $x_i=1$, $i=1,\dots,d$, whose standardized notation is $\zeta^\star(s_1,\dots,s_d)$, i.e.
\[
\Li^\star_{s_1,\dots,s_d}(1,\dots,1)
	=:\zeta^\star(s_1,\dots,s_d)
\]
with $s_i\in\N$ and $s_1>1$, to ensure the convergence. The values $\zeta^\star(s_1,\dots,s_d)$ are called {\itshape multiple zeta-star values} and are intensively studied due to their interesting mathematical properties, see Zhao~\cite{zhao} and Eie~\cite{eie.book} for an overview, and applications in physics, see Smirnov~\cite{smirnov} and Weinzierl~\cite{weinzierl}. The definition of these sums goes back to Hoffman~\cite{hoffman.pjm} and Zagier~\cite{zagier} who invented these values independently around 1992.
\end{remark}

One of our applications deduced in \cite{gencev.dm} led to the following new relation.

\begin{theorem}[{\cite[Theorem 2.2]{gencev.dm}}]
\label{th.series}
Assume that $s\in\mathbb N$, $s\ge 2$, and $a,p\in\mathbb R$ with $p\neq 1$. Then
\begin{equation}
\label{eq:series}
\Li_s(a)
	=\Li^\star_{\{1\}_s}\bigl(1-p,\{1\}_{s-2},1+\tfrac{ap}{1-p}\bigr)
	-\Li^\star_{\{1\}_s}\bigl(1-p,\{1\}_{s-1}\bigr),
\end{equation}
where $|a|\le\min\bigl(1,2/p-1\bigr)$, and $\{1\}_r$ stands for $1$ repeated $r$ times.
\end{theorem}

The identity in \eqref{eq:series}, coming from the limiting case of \eqref{eq:gencev} for $n\to\infty$, is appealing from two aspects. First, by temporarily suppressing the discrete parameter $s$, its right-hand side involves two real variables $a,p$, whereas its left-hand side involves only one of them, the~$a$. Second, the computational complexity of both sides is essentially different. While the left-hand side of \eqref{eq:series} is represented by a one-dimensional polylogarithm, the right-hand side is in the form of a difference of $s$-dimensional polylogarithms. As a by-product, we~also obtained in \cite[Corollary~2.3]{gencev.dm} the reductions of the multiple polylogarithms on the right-hand side of~\eqref{eq:series} to one-dimensional polylogarithms in the form
\begin{align*}
\Li^\star_{\{1\}_s}\bigl(1-p,\{1\}_{s-1}\bigr)
	&=-\Li_s\bigl(1-\tfrac1p\bigr),\\[.5em]
\Li^\star_{\{1\}_s}\bigl(1-p,\{1\}_{s-2},1+\tfrac{ap}{1-p}\bigr)
	&=\Li_s(a)
	 -\Li_s\bigl(1-\tfrac1p\bigr),
\end{align*}
where $a,p\in\R$ with $|a|\le 1$, $1/2\le p<1$, and $s\in\N$ with $s\ge 2$. Hence, the last two relations allow us to calculate high-precision approximations of the multiple polylogarithms on the left-hand sides with the help of their one-dimensional `equivalents'.

\subsection{Conjecture of Pan and Xu}

The subject of Mneimneh's sums and their various modifications have dramatically developed in recent times due to the essential popularity of harmonic sums. Just recently, Pan and Xu~\cite{pan.xu} stated a conjecture related to specific Mneimneh-like sums, where the standard harmonic numbers $H_k$ are replaced by much more general objects $\zeta_k^\star(s_1,\dots,s_k)$ defined by
\[
\zeta_k^\star(s_1,\dots,s_d)
	:=\sum_{k\ge n_1\ge\cdots\ge n_d\ge 1}
		\frac1{n_1^{s_1}\cdots n_d^{s_d}}
\]
and called {\itshape multiple harmonic-star values}. Here, $s_i$ are assumed to be positive integers.
\medskip

\noindent
{\bfseries Conjecture (Pan, Xu~\cite{pan.xu})}
{\itshape
Assume that $r\in\N_0$, $n\in\N$, $x,y\in\R$, $x\neq -y$. Furthermore, let $\bm u=(u_1,\dots,u_{r+1})\in\N_0^{r+1}$, $\bm m=(m_1,\dots,m_r)\in\N_0^r$, and set 
\[
|\bm u|_j:=\sum_{i=1}^ju_i,
\qquad\qquad
|\bm m|_j:=\sum_{i=1}^jm_i.
\]
Then}
\begin{align*}
\sum_{k=1}^n&\binom nk\cdot x^ky^{n-k}\cdot
	\zeta^\star_k
	\left(%
		\{1\}_{u_1},m_1+2,\dots,\{1\}_{u_r},m_r+2,\{1\}_{u_{r+1}}
	\right)
\\[.5em]
&=(x+y)^n\cdot
	\sum_{n\ge n_1\ge\cdots\ge n_{|\bm u|_{r+1}+|\bm m|_r+r-1}\ge 1}
	\frac{%
		\bigl(\frac y{x+y}\bigr)^{%
			\sum_{j=1}^r
			\bigl(
				n_{|\bm u|_j+|\bm m|_{j-1}+j-1}-n_{|\bm u|_j+|\bm m|_j+j}
			\bigr)}}{%
		n_1\cdots n_{|\bm u|_{r+1}+|\bm m|_r+r-1}}\times\\
	&\hskip7cm
		\times\left(1-\left(\frac y{x+y}\right)^{n_{|\bm u|_{r+1}+|\bm m|_r+r-1}}\right).
\end{align*}
\medskip

Although complex, the above formula can be simplified if divided by $(x+y)^n$. After that, one can introduce the variable $p:=x/(x+y)$, and consequently, the left-hand side of the conjectured formula becomes $\sum_{k=1}^n\binom nk\cdot p^k\cdot(1-p)^{n-k}\cdot\zeta^\star_k(s_1,\dots,s_d)$ for suitable arguments~$s_i$ according to the left-hand side of the conjectured formula. Similarly as in the formula~\eqref{eq:gencev}, we observe that the above conjecture is related to partial sums of multiple polylogarithms and, therefore, of a~hypothetical significance.

\subsection{Objectives and paper outline}

The base for the research conducted in this work consists of sums that are slightly more general than those introduced by Pan and Xu. In fact, we focus on the properties of the sums
\begin{equation}
\label{eq:xmneimneh.ap.def}
M_n^{(s_1,\dots,s_d)}(a,p)
	:=\sum_{k=1}^n\binom nk\cdot p^k\cdot (1-p)^{n-k}\cdot\zeta_k^\star(s_1,\dots,s_d;a),
\end{equation}
where $s_i\in\N$, $a,p\in\R$, and
\[
\zeta_k^\star(s_1,\dots,s_d;a)
	:=\sum_{k\ge n_1\ge\cdots\ge n_d\ge 1}
		\frac{a^{n_d}}{n_1^{s_1}\cdots n_d^{s_d}}
\]
denotes the {\itshape generalized multiple harmonic-star value}.
\goodbreak

In the following paragraphs, we outline three particular goals of this paper.

\subsubsection*{Goal 1 -- Transformation of Mneimneh-like sums}

Our first goal is to establish a generalized variant of identity \eqref{eq:gencev} relating to the `higher' Mneimneh-like sums $M_n^{(s_1,\ldots,s_d)}(a,p)$. Our corresponding result presented in Theorem~\ref{th:main} encompasses the conjecture of Pan and Xu and extends its statement due to the introduced parameter $a$. Even if stating Theorem~\ref{th:main} can be perceived as a pure mathematical vagary consisting of writing one sum to a new one (even more complicated, see the right-hand side of~\eqref{eq:gencev} as well as the conjecture of Pan and Xu), our main impetus for studying these sums lies in finding a generalization of Theorem~\ref{th.series} providing a link to multiple polylogarithms. Therefore, Theorem~\ref{th:main} is one of our main results.

\subsubsection*{Goal 2 -- Polylogarithm relations}

As just anticipated, our second goal is to provide applications of Theorem~\ref{th:main} to multiple polylogarithms. We show in Sect.~\ref{sec:polylog} that the transformed form of $M_n^{(s_1,\dots,s_d)}(a,p)$, see~\eqref{eq:main}, is still suitable for studying properties of specific multiple polylogarithms whose understanding is currently considered as highly incomplete. We present the complex Theorems~\ref{th:Li.1} and~\ref{th:Li.2} aiming to show that the value $\Li_{\bm s}^\star(\{1\}_{\|\bm s\|-1},a)$, where $\|\boldsymbol s\|:=\dim\boldsymbol s$, can be written as a~difference of another multiple polylogarithms containing one more additional parameter~$p$, not involved in $\Li_{\bm s}^\star(\{1\}_{\|\bm s\|-1},a)$. To present a very simple instance, Example~\ref{ex:Li.1} (based on Theorem~\ref{th:Li.1} with $a=1$) implies the somewhat strange formula
\[
\Li^\star_{1,1,1,1}\bigl(1-p,\tfrac1{1-p},1-p,\tfrac1{1-p}\bigr)
-\Li^\star_{1,1,1,1}\bigl(1-p,\tfrac1{1-p},1-p,1\bigr)
=\tfrac{7}{360}\,\pi^4,
\]
where $p\in(0,1)$ is arbitrary. Nevertheless, our investigations of Mneimneh's generalized sums allow us to uncover new general properties of multiple polylogarithms. For example, we present the reduction formula, see Corollary~\ref{cr:Li.11},
\[
\Li^\star_{\{1\}_{|\bm s|}}
		\left(
			\textstyle
			\bigsqcup_{i=1}^{d-1}
			\bigl\{
				1-p,\{1\}_{m_i},\frac1{1-p},\{1\}_{u_i}
			\bigr\},
			1-p,\{1\}_{m_d+1}
		\right)
  =-\Li^\star_{\bm s}
		\left(
			\{1\}_{\|\bm s\|-1},1-\tfrac1p
		\right),
\]
where $\sqcup$ stands for concatenation of strings, $|\boldsymbol s|=s_1+\cdots+s_d$ denotes the sum of all components of $\boldsymbol s$, and $\|\boldsymbol s\|=\dim\boldsymbol s$. We refer readers to Sect.~\ref{sec:polylog} for more general scenarios and precise assumptions.

\subsubsection*{Goal 3 -- Arithmetic means of multiple harmonic numbers}

The third goal extends the described theory in a different sense. Our general result presented in Theorem~\ref{th:mean} provides a transformation of the arithmetic mean of the elements $\zeta_k^\star(\bm s;a)$, $k=0,1,\dots,n$. The corollaries of Theorem~\ref{th:mean} include new representations of arbitrary zeta\discretionary{-}{-}{-}star values $\zeta^\star(\bm s)$ or, more generally, of arbitrary convergent values $\Li_{\bm s}^\star(\{1\}_{\|\bm s\|-1},a)$, see Corollaries~\ref{cr:mean.1} and~\ref{cr:mean.2}.

\section{Transformation of Mneimneh-like sums}
\label{sec:transf}

\subsection{Statement and examples}

In this section, we state the main transformation theorem allowing us to rewrite the Mneimneh\discretionary{-}{-}{-}like sum $M_n^{(s_1,\dots,s_d)}(a,p)$ to a partial sum of a specific multiple polylogarithm. The following statement forms the technical base for our applications, which are described in the rest of this paper.

\begin{theorem}
\label{th:main}
Assume that $\bm s:=(s_1,\dots,s_d)\in\mathbb N^d$, $a,p\in\mathbb R$, and $d\in\N$. Then
\begin{equation}
\label{eq:main}
M_n^{(\bm s)}(a,p)
	=\sum_{n\ge n_1\ge\cdots\ge n_{|\bm s|}\ge 1}
	 \frac{\prod_{r=1}^d(1-p)^{n_{|\bm s|_{r-1}+1}-n_{|\bm s|_r}}}{n_1\cdots n_{|\bm s|}}\cdot
	 \Bigl[
		(1-p+ap)^{n_{|\bm s|}}
	 -(1-p)^{n_{|\bm s|}}
	 \Bigr],
\end{equation}
where $|\bm s|_r:=\sum_{i=1}^rs_i$ and $|\bm s|:=|\bm s|_d$.
\end{theorem}

Before approaching the proof of Theorem~\ref{th:main}, which is postponed to Sect.~\ref{ssec:proof.main}, we present several more straightforward consequences to familiarize the statement.

\begin{example}
\label{ex:first}
We have
\[
\sum_{k=1}^n\binom nk\cdot
	\sum_{j=1}^k\frac{\sum_{i=1}^j\frac{(-1)^{i-1}}{i^2}}{j^3}
	=\sum_{n\ge n_1\ge\cdots\ge n_5\ge 1}
		\frac{2^{n-n_1+n_3-n_4}}{n_1\cdots n_5}.
\]
\end{example}

\begin{proof}[of Example~\ref{ex:first}]
The sum on the left-hand side in the example agrees with $M_n^{(\bm s)}(a,p)$ with $\bm s=(3,2)$, $a=-1$, and $p=1/2$. Therefore, $|\bm s|_1=s_1=3$ and $|\bm s|_2=s_1+s_2=5$. Consequently, $\prod_{r=1}^2(1-p)^{n_{|\bm s|_{r-1}+1}-n_{|\bm s|_r}}=(1-p)^{n_1-n_3}\cdot(1-p)^{n_4-n_5}$. By setting all the given parameters into \eqref{eq:main}, we obtain
\[
\sum_{k=1}^n
	\binom nk\cdot
	\left(\tfrac12\right)^k\cdot\left(\tfrac12\right)^{n-k}\cdot
	\zeta^\star_k(3,2;-1)
 =\sum_{n\ge n_1\ge\cdots\ge n_5\ge 1}
 	\frac{%
		\left(\frac12\right)^{n_1-n_3}\cdot
		\left(\frac12\right)^{n_4-n_5}}{n_1\cdots n_5}\cdot
	\left[
		0-\left(\tfrac12\right)^{n_5}
	\right].
\]
Simplifying the last relation concludes the proof.\qed
\end{proof}

\begin{example}
\label{ex:Mn.112}
Assume that $a,p\in\R$, $p\neq 1$, and $d,m,n\in\N$. Then
\begin{align*}
M_n^{(\{m\}_d)}(a,p)
 &=\sum_{n\ge n_1\ge\cdots\ge n_{md}\ge 1}
	 \frac{\prod_{r=1}^d(1-p)^{n_{mr-m+1}-n_{mr}}}{n_1\cdots n_{md}}\cdot
	\bigl[(1-p+ap)^{n_{md}}-(1-p)^{n_{md}}\bigr],
\end{align*}
where $\{m\}_d$ denotes the argument $m$ repeated $d$ times.
\end{example}

\begin{proof}[of Example~\ref{ex:Mn.112}]
Set $\bm s=(\{m\}_d)$ in Theorem~\ref{th:main}. Then $|\bm s|=md$ and $|\bm s|_r=mr$ for every $r=0,\dots,d$. Therefore, $\prod_{r=1}^d(1-p)^{n_{|\bm s|_{r-1}+1}-n_{|\bm s|_r}}=\prod_{r=1}^d(1-p)^{n_{mr-m+1}-n_{mr}}$. Substituting this form into \eqref{eq:main} concludes the proof.
\qed
\end{proof}

The formula in Example~\ref{ex:Mn.112} simplifies with $m=1$. In this particular instance, we deduce
\begin{equation}
\label{eq:Mn.1}
M_n^{(\{1\}_d)}(a,p)
 =\sum_{n\ge n_1\ge\cdots\ge n_d\ge 1}
	 \frac{(1-p+ap)^{n_d}-(1-p)^{n_d}}{n_1\cdots n_d}.
\end{equation}
We also point out that by setting $a=d=1$ in \eqref{eq:Mn.1}, the identity reduces to Mneimneh's original result~\eqref{eq:mneimneh}. Additionally, by dividing \eqref{eq:Mn.1} by $(1-p)^n$ with $p\to\infty$, we obtain
\begin{equation}
\label{eq:dilcher+}
\sum_{k=1}^n\binom nk\cdot (-1)^k\cdot\zeta^\star_k(\{1\}_d;a)
	=\sum_{\substack{n\ge n_1\ge\cdots\ge n_d\ge 1\\[.2em]n_d=n}}
	 \frac{(1-a)^{n_d}-1}{n_1\cdots n_d}
	=\frac{(1-a)^n-1}{n^d}.
\end{equation}
For example, the last identity with $a=2$ implies the somewhat curious formula
\begin{equation}
\label{eq:2nd.cases}
\sum_{k=1}^n\binom nk\cdot (-1)^{k-1}\cdot\zeta^\star_k(\{1\}_d;2)
	=\begin{cases}
		0						&\text{$n$ even},\\
		\frac2{n^d}	&\text{$n$ odd},
	 \end{cases}
\end{equation}
which enables us to arrive at the following corollary.

\begin{corollary}
\label{cr:dilcher-like}
For every $d,n\in\N$, we have
\[
\sum_{k=1}^{\left\lfloor\frac{n+1}2\right\rfloor}
	\binom n{2k-1}\cdot\frac1{(2k-1)^d}
	=\frac{\zeta^\star_n(\{1\}_d;2)}2.
\]
\end{corollary}

\begin{proof}
Recall the binomial inverse relation, i.e.~if $a_n=\sum_{k=0}^n\binom nk\cdot(-1)^k\cdot b_k$, then $b_n=\sum_{k=0}^n\binom nk\cdot(-1)^k\cdot a_k$, see e.g.~Riordan~\cite[p.~43]{riordan}. Applying this principle to~\eqref{eq:2nd.cases} with $b_k=\zeta^\star_k(\{1\}_d;2)$ and $a_n$ the right-hand side of \eqref{eq:2nd.cases}, we immediately deduce the sought identity.
\qed
\end{proof}

The above relation is connected with the so-called Dilcher's formula (see Dilcher~\cite[p.~93]{dilcher} or Nica~\cite[Eq.~(3)]{nica}) stating that
$
\sum_{k=1}^n\binom nk\cdot(-1)^{k-1}/k^d
	=\zeta_n^\star(\{1\}_d).
$
Corollary~\ref{cr:dilcher-like} provides its odd counterpart. The even counterpart with $\sum_{k=1}^{\lfloor n/2\rfloor}\binom n{2k}/(2k)^d$ can then easily be obtained as a~complement. Furthermore, setting $a=1$ in~\eqref{eq:dilcher+} and applying the principle of the binomial inversion lead to Dilcher's identity directly.

We finish our initial considerations by the following statement, enabling us to exclude uninteresting cases from the rest of the paper.
\begin{lemma}
\label{lm:p01}
For $p=0$ and $p=1$, the summation transformation \eqref{eq:main} is trivial.
\end{lemma}

\begin{proof}
For $p=0$, the statement immediately follows from the definition of $M_n^{(\bm s)}(a,p)$ and~\eqref{eq:main}. For $p=1$, the situation requires more attention. First of all, the definition of $M_n^{(\bm s)}(a,p)$ in~\eqref{eq:xmneimneh.ap.def} implies that $M_n^{(\bm s)}(a,1)=\zeta^\star_n(\bm s;a)$. Next, we concentrate on the right-hand side of \eqref{eq:main} assuming that $\bm s=(s_1,\dots,s_d)$. Since
\[
(1-p)^{n_{|\bm s|_{r-1}+1}-n_{|\bm s|_r}}\Bigr |_{p=1}
	=\begin{cases}
		0&\text{for $n_{|\bm s|_{r-1}+1}>n_{|\bm s|_r}$},\\[.5em]
		1&\text{for $n_{|\bm s|_{r-1}+1}=n_{|\bm s|_r}$},
	 \end{cases}
\]
this and \eqref{eq:main} yield
\[
M_n^{(\bm s)}(a,1)
 =\sum_{%
		\substack{%
 			n\ge n_1\ge\cdots\ge n_{|\bm s|}\ge 1\\[.2em]
			n_1=n_{|\bm s|_1}\\[.2em]
			n_{|\bm s|_1+1}=n_{|\bm s|_2}\\
			\vdots\\
			n_{|\bm s|_{d-1}+1}=n_{|\bm s|}
			}}
	 \frac{a^{n_{|\bm s|}}}{n_1\cdots n_{|\bm s|}}
 =\sum_{n\ge j_1\ge\cdots\ge j_d\ge 1}
	 \frac{a^{j_d}}{j_1^{s_1}\cdots j_d^{s_d}}
 =\zeta^\star_n(\bm s;a),
\]
where we set $j_r:=n_{|\bm s|_{r-1}+1}=\cdots=n_{|\bm s|_r}$, $r=1,\dots,d$.
\qed
\end{proof}

\subsection{Proof of Theorem~\ref{th:main}}
\label{ssec:proof.main}

We precede the proof of Theorem~\ref{th:main} by the following auxiliary lemma, which will be applied several times within this section.

\begin{lemma}
\label{lm:aux}
Assume that $a,x\in\R$ and $n\in\N$ are arbitrary. Then
\begin{equation}
\label{eq:aux.1}
\int_0^a\frac{(1+Ax)^n-1}A\,\mathrm dA
	=\sum_{j=1}^n\frac{(1+ax)^j-1}j.
\end{equation}
\end{lemma}

\begin{proof}
The statement is clear for $x=0$. Therefore, we further assume that $x\neq 0$. Under this condition, we can write
\[
\int_0^a\frac{(1+Ax)^n-1}A\,\mathrm dA
	=x\cdot\int_0^a\frac{(1+Ax)^n-1}{(1+Ax)-1}\,\mathrm dA
	=x\cdot\sum_{j=1}^n\int_0^a(1+Ax)^{j-1}\,\mathrm dA.
\]
Finally, performing the integration concludes the proof.\qed
\end{proof}

Based on Lemma~\ref{lm:aux}, we observe that the similar identity
\begin{equation}
\label{eq:aux.2}
\int_{1-a}^1\frac{(1+Ax)^n-1}A\,\mathrm dA
	=\sum_{j=1}^n\frac{(1+x)^j-(1+x-ax)^j}j
\end{equation}
holds under identical assumptions as in Lemma~\ref{lm:aux}. Of course, for the proof of~\eqref{eq:aux.2}, we apply~\eqref{eq:aux.1} twice. First, with $a=1$, and second, with $1-a$ instead of $a$. This and $\int_{1-a}^1=\int_0^1-\int_0^{1-a}$ then imply \eqref{eq:aux.2} playing an important role in the proof of Theorem~\ref{th:main}.

\begin{remark}
Since the transformation effect of Theorem~\ref{th:main} for $p=1$ is negligible by Lemma~\ref{lm:p01}, we assume $p\neq 1$ in the following proof. Even if the same can be stated concerning $p=0$, this case does not cause technical complications contrary to $p=1$.
\end{remark}

\begin{proof}[of Theorem~\ref{th:main}]
The proof of Theorem~\ref{th:main} can be deduced with the help of the formula in~\eqref{eq:gencev} and with the induction. To enable more lucidity, we split the proof into several parts.
\begin{enumerate}
\item Let us first consider the case $d=1$ in~\eqref{eq:main}. Then $\bm s=(s_1)$, $|\bm s|=|\bm s|_1=s_1$ and, consequently, the identity in \eqref{eq:main} reads
\[
M_n^{(s_1)}(a,p)
	=\sum_{n\ge n_1\ge\cdots\ge n_{s_1}\ge 1}
	 \frac{(1-p)^{n_1-n_{s_1}}}{n_1\cdots n_{s_1}}\cdot
	 \Bigl[
		(1-p+ap)^{n_{s_1}}-(1-p)^{n_{s_1}}
	 \Bigr],
\]
which agrees with \eqref{eq:gencev} with $s=s_1$.
\item Now, assume that \eqref{eq:main} is true for arbitrary $\bm s=(s_1,\dots,s_d)\in\mathbb N^d$, $d\in\N$. We show that this implies Theorem~\ref{th:main} also for the $(d+1)$-tuple $\bm t=(s_1,\dots,s_d,1)$. To see this, we first generate the corresponding Mneimneh-like sum as follows:
\begin{align}
\label{eq:diff.M}
\int_{1-A}^1
 &\Bigl(
		M_n^{(s_1,\dots,s_d)}(1,p)
	 -M_n^{(s_1,\dots,s_d)}(1-a,p)
	\Bigr)\cdot\frac{\mathrm da}a
\\[.5em]
\notag
 &=\sum_{k=1}^n\binom nk\cdot p^k\cdot (1-p)^{n-k}
	 \cdot\int_{1-A}^1\frac{%
		\zeta_k^\star(s_1,\dots,s_d;1)
	 -\zeta_k^\star(s_1,\dots,s_d;1-a)}a\,\mathrm da
\\[.5em]
\notag
 &=\sum_{k=1}^n\binom nk\cdot p^k\cdot (1-p)^{n-k}
	 \sum_{k\ge n_1\ge\cdots\ge n_d\ge 1}
	 \frac 1{n_1^{s_1}\cdots n_d^{s_d}}\cdot
	 \int_{1-A}^1\frac{1-(1-a)^{n_d}}a\,\mathrm da
\\[.5em]
\notag
 &=\sum_{k=1}^n\binom nk\cdot p^k\cdot (1-p)^{n-k}
	 \sum_{k\ge n_1\ge\cdots\ge n_d\ge 1}
	 \frac 1{n_1^{s_1}\cdots n_d^{s_d}}\cdot
	 \zeta_{n_d}^\star(1;A)
\\[.5em]
\label{eq:Mn,1}
 &=M_n^{(s_1,\dots,s_d,1)}(A,p),
\end{align}
where the evaluation of the last integral follows form~\eqref{eq:aux.2}. Next, we apply the inductive assumption, i.e.~formula~\eqref{eq:main}, to the difference in~\eqref{eq:diff.M}. By these facts and by~\eqref{eq:aux.2} and~\eqref{eq:Mn,1}, we obtain
\begin{align*}
M_n^{(s_1,\dots,s_d,1)}&(A,p)\\[.5em]
 &=\int_{1-A}^1
  \Bigl(
		M_n^{(s_1,\dots,s_d)}(1,p)
	 -M_n^{(s_1,\dots,s_d)}(1-a,p)
	\Bigr)\cdot\frac{\mathrm da}a
\\[.5em]
 &=\sum_{n\ge n_1\ge\cdots\ge n_{|\bm s|}\ge 1}
	 \frac{\prod_{r=1}^d(1-p)^{n_{|\bm s|_{r-1}+1}-n_{|\bm s|_r}}}{n_1\cdots n_{|\bm s|}}\cdot
	 \int_{1-A}^1\frac{1-(1-ap)^{n_{|\bm s|}}}a\,\mathrm da
\\[.5em]
 &=\sum_{n\ge n_1\ge\cdots\ge n_{|\bm s|}\ge 1}
	 \frac{\prod_{r=1}^d(1-p)^{n_{|\bm s|_{r-1}+1}-n_{|\bm s|_r}}}{n_1\cdots n_{|\bm s|}}\cdot
	 \sum_{j=1}^{n_{|\bm s|}}\frac{(1-p+Ap)^j-(1-p)^j}j
\\[.5em]
 &=\sum_{n\ge n_1\ge\cdots\ge n_{|\bm t|}\ge 1}
	 \frac{\prod_{r=1}^{d+1}(1-p)^{n_{|\bm t|_{r-1}+1}-n_{|\bm t|_r}}}{n_1\cdots n_{|\bm t|}}\cdot
	 \Bigl[
		(1-p+Ap)^{n_{|\bm t|}}
	 -(1-p)^{n_{|\bm t|}}
	 \Bigr],
\end{align*}
which agrees with \eqref{eq:main} with $\bm t=(s_1,\dots,s_d,1)$ instead of $\bm s$.
\item Finally, we show that Theorem~\ref{th:main} holds also for the $(d+1)$-tuple $\bm v=(s_1,\dots,s_d,s_{d+1})\in\N^{d+1}$ with $s_{d+1}\ge 2$. The iterative integration generates the corresponding Mneimneh-like sum according to the relation
\[
M_n^{(s_1,\dots,s_{d+1})}(a,p)
	=\idotsint\limits_{0<A_1<\cdots<A_{s_{d+1}-1}<a}
	 M_n^{(s_1,\dots,s_d,1)}(A_1,p)\cdot
	 \frac{\mathrm dA_1\cdots\mathrm dA_{s_{d+1}-1}}{A_1\cdots A_{s_{d+1}-1}}.
\]
We proceed by applying the formula for the Mneimneh-like sum $M_n^{(s_1,\dots,s_d,1)}(A,p)$ deduced at the end of the second part of this proof. We infer for $p\neq 1$ that
\begin{align*}
M_n^{(s_1,\dots,s_{d+1})}(a,p)
 &=\sum_{n\ge n_1\ge\cdots\ge n_{|\bm t|}\ge 1}
	 \frac{\prod_{r=1}^{d+1}(1-p)^{n_{|\bm t|_{r-1}+1}-n_{|\bm t|_r}}}{n_1\cdots n_{|\bm t|}}\cdot(1-p)^{n_{|\bm t|}}\times{}\\[.5em]
 &\hskip1cm{}\times
	\idotsint\limits_{0<A_1<\cdots<A_{s_{d+1}-1}<a}
	\left[
		\left(1+\frac{A_1p}{1-p}\right)^{n_{|\bm t|}}-1
	\right]
	\cdot
	\frac{\mathrm dA_1\cdots\mathrm dA_{s_{d+1}-1}}{A_1\cdots A_{s_{d+1}-1}}.
\intertext{Next, employing Lemma~\ref{lm:aux} on the last multiple integral successively, we arrive at}
M_n^{(s_1,\dots,s_{d+1})}(a,p)
 &=\sum_{n\ge n_1\ge\cdots\ge n_{|\bm t|}\ge 1}
	 \frac{\prod_{r=1}^{d+1}(1-p)^{n_{|\bm t|_{r-1}+1}-n_{|\bm t|_r}}}{n_1\cdots n_{|\bm t|}}\cdot(1-p)^{n_{|\bm t|}}\times{}\\[.5em]
 &\hskip2cm{}\times
	\sum_{n_{|\bm t|}\ge j_1\ge\cdots\ge j_{s_{d+1}-1}\ge 1}
	\frac{\left(1+\frac{ap}{1-p}\right)^{j_{s_{d+1}-1}}-1}{j_1\cdots j_{s_{d+1}-1}}
\\[2em]
 &=\sum_{n\ge n_1\ge\cdots\ge n_{|\bm v|}\ge 1}
	 \frac{\prod_{r=1}^{d+1}(1-p)^{n_{|\bm v|_{r-1}+1}-n_{|\bm v|_r}}}{n_1\cdots n_{|\bm v|}}\times\\
 &\hskip2cm{}\times
	 \Bigl[
		(1-p+pa)^{n_{|\bm v|}}
	 -(1-p)^{n_{|\bm v|}}
	 \Bigr],
\end{align*}
which agrees with \eqref{eq:main} with $\bm v=(s_1,\dots,s_d,s_{d+1})$ instead of $\bm s$.\qed
\end{enumerate}
\end{proof}

\section{Applications to multiple polylogarithms}
\label{sec:polylog}

We now approach applications of Theorem~\ref{th:main} to multiple polylogarithms. The section will be split into two subsections, allowing us to deal with two technical cases whose separation seems more convenient. We investigate the consequences of Theorem~\ref{th:main} for the tuples
$\bm s
=\left(
	\textstyle\sqcup_{i=1}^{d-1}\{m_i+2,\{1\}_{u_i}\},m_d+2
 \right)$
in the first part, whereas for the tuples
$\bm s
=\left(
	\textstyle\sqcup_{i=1}^d\{m_i+2,\{1\}_{u_i}\}
 \right)$ with $u_d\ge 1$
in the second part. Here, the parameters $m_i$, $u_i$ are assumed to be non-negative integers, and $\sqcup$ denotes the concatenation of strings.

\subsection{First theorem on polylogarithms}
\label{ssec:polylog.1}

We present our first result with multiple polylogarithms below. Notice that the right-hand side of \eqref{eq:Li.10} involves two independent parameters~$a,p$, whereas the left-hand side involves $a$ only.

\begin{theorem}
\label{th:Li.1}
Assume that $m_i,u_i\in\N_0$, $d\in\N$, and
$
\bm s
	=\left(
		\textstyle\sqcup_{i=1}^{d-1}\{m_i+2,\{1\}_{u_i}\},m_d+2
	 \right).
$
If $a,p\in\R$ with $p\neq 1$ and $|a|\le \min\bigl(1,2/p-1\bigr)$ then
\begin{align}
\notag
\Li^\star_{\bm s}
	\left(
		\{1\}_{\|\bm s\|-1},a
	\right)
 &=\Li^\star_{\{1\}_{|\bm s|}}
		\left(
			\textstyle
			\bigsqcup_{i=1}^{d-1}
			\bigl\{
				1-p,\{1\}_{m_i},\frac1{1-p},\{1\}_{u_i}
			\bigr\},
			1-p,\{1\}_{m_d},1+\tfrac{ap}{1-p}
		\right)\\[.5em]
\label{eq:Li.10}
 &\,
  -\Li^\star_{\{1\}_{|\bm s|}}
		\left(
			\textstyle
			\bigsqcup_{i=1}^{d-1}
			\bigl\{
				1-p,\{1\}_{m_i},\frac1{1-p},\{1\}_{u_i}
			\bigr\},
			1-p,\{1\}_{m_d+1}
		\right).
\end{align}
\end{theorem}

Consequences of this theorem are given later within this section, see Corollaries~\ref{cr:Li.1},~\ref{cr:Li.11} and Example~\ref{ex:Li.1}. Let us now approach the proof of Theorem~\ref{th:Li.1}, based on the Toeplitz theorem. To make the paper reading comfortable, we include this crucial result below.

\begin{theorem}[Toeplitz Theorem, {Knopp~\cite[p.~74]{knopp}}]
\label{th.toeplitz}
Assume that \smash{$\lim\limits_{n\to\infty}x_n=0$} and the terms~$a_{nk}$
form a~triangular system
\[
\left\{
	\begin{matrix}[ccccc]
	a_{00}\\
	a_{10}& a_{11}\\
	\vdots& \vdots& \ddots\\
	a_{n0}& a_{n1}& \cdots & a_{nn}\\
	\vdots& \vdots& \vdots& \vdots& \ddots\\
	\end{matrix}
\right.
\]
and satisfy the following conditions:
\begin{enumerate}[label={\textnormal{(\alph*)}}]
\item
	every column of the system contains a null sequence, i.e.~for fixed $k\ge 0$,
	$\smash{\lim\limits_{n\to\infty}}a_{nk}=0$,
\item
	there exists a constant $K$ independent of $n$ such that
	$
	\sum_{k=0}^n |a_{nk}|<K.
	$
\end{enumerate}
Then
\begin{equation}
\label{eq:toplitz}
\lim_{n\to\infty}\sum_{k=0}^na_{nk}\cdot x_k=0.
\end{equation}
\end{theorem}

Concerning our later technical needs, we append the following practical remark relating to the Toeplitz theorem.

\begin{remark}
\label{rm:toeplitz}%
Assume that $\{x_k'\}$ is a convergent real sequence with $x_k'\to\xi$ for $k\to\infty$. Further, set $a_{nk}=1/(n+1)$ for every $k=0,\dots,n$ and $x_k:=x_k'-\xi$ in the Toeplitz theorem. Then all its conditions are satisfied and, consequently, the arithmetic mean $1/(n+1)\cdot\sum_{k=0}^n(x'_k-\xi)$ converges to $0$ for $n\to\infty$. From this, we immediately deduce the well-known fact that $1/(n+1)\cdot\sum_{k=0}^nx'_k\to\xi$.
\end{remark}

\begin{proof}[of Theorem~\ref{th:Li.1}]
Set
\[\bm s
=\left(
	\textstyle\bigsqcup_{i=1}^{d-1}\{m_i+2,\{1\}_{u_i}\},m_d+2
 \right)
=(s_1,\dots,s_{\|\bm s\|})\]
in Theorem~\ref{th:main}. Notice that the product factor $(1-p)^{n_{|\bm s|_{r-1}+1}-n_{|\bm s|_r}}$ involved in~\eqref{eq:main} is equal to~$1$ for $|\bm s|_{r-1}+1=|\bm s|_r$ iff $s_r=1$ for some $r=1,\dots,\|\bm s\|$. These cases are related to the groups of $1$'s in $\bm s$ (if any), i.e.~to the elements $\{1\}_{u_i}$ with $u_i\in\N_0$. Hence, every $1$ from the group $\{1\}_{u_i}$ contributes to the product $\prod_{r=1}^{\|\bm s\|}(1-p)^{n_{|\bm s|_{r-1}+1}-n_{|\bm s|_r}}$ with exactly one factor equal to~$1$. These trivial factors are formally included as~$1^{u_i}$ in the following calculation. Denoting $\bm m=(m_1,\dots,m_d)$ and $\bm u=(u_1,\dots,u_{d-1})$, we can write for our choice of $\bm s$ (declared at the beginning of this proof) that
\begin{align*}
\prod_{r=1}^{\|\bm s\|}(1-p)^{n_{|\bm s|_{r-1}+1}-n_{|\bm s|_r}}
 &=\quad(1-p)^{n_1}\cdot
	 \bigl(\tfrac1{1-p}\bigr)^{n_{m_1+2}}\cdot
	 1^{u_1}\times{}\\[.5em]
 &\quad{}\times
	 (1-p)^{n_{m_1+u_1+3}}\cdot
	 \bigl(\tfrac1{1-p}\bigr)^{n_{m_1+m_2+u_1+4}}\cdot
	 1^{u_2}\times{}\\
 &\quad\cdots\\
 &\quad{}\times
	 (1-p)^{n_{|\bm m|_{d-2}+|\bm u|_{d-2}+2d-3}}\cdot
	 \bigl(\tfrac1{1-p}\bigr)^{n_{|\bm m|_{d-1}+|\bm u|_{d-2}+2d-2}}\cdot
	 1^{u_{d-1}}\times{}\\[.5em]
 &\quad{}\times
	 (1-p)^{n_{|\bm m|_{d-1}+|\bm u|_{d-1}+2d-1}}\cdot
	 \bigl(\tfrac1{1-p}\bigr)^{n_{|\bm m|_d+|\bm u|_{d-1}+2d}}\\[.5em]
 &=\prod_{r=1}^d
		\left[
			(1-p)^{n_{|\bm m|_{r-1}+|\bm u|_{r-1}+2r-1}}\cdot\bigl(\tfrac1{1-p}\bigr)^{n_{|\bm m|_r+|\bm u|_{r-1}+2r}}
		\right],
\end{align*}
where we intentionally omitted the trivial factors $1^{u_i}$ in the last step. The building mechanism for the indices $n_{\alpha}$ in the above calculation is demonstrated for $\bm s=(4,1,1,3,2)$ in Figure~\ref{fg:example} to enable more clarity.
\begin{figure}[h]
\begin{center}
\begin{tikzpicture}[font=\scriptsize,remember picture,x=0.9cm,y=0.9cm]
\draw[-stealth]
	(0.5,0) -- (13,0)
		node[below] {$\alpha$};
\node[font=\normalsize] (formula) at (7,-3)
	{$\displaystyle
		(1-p)^{n_1}\cdot
		\bigl(\tfrac1{1-p}\bigr)^{n_4}\cdot
		(1-p)^{n_7}\cdot
		\bigl(\tfrac1{1-p}\bigr)^{n_9}\cdot
		(1-p)^{n_{10}}\cdot
		\bigl(\tfrac1{1-p}\bigr)^{n_{11}}
		$};
\foreach\n in {1,...,10}
	\draw[thin]
		(\n,-.05) -- (\n,.05)
			node[below,yshift=-5pt] {$\n$};
\draw[thin]
	(11,-.05) -- (11,.05)
		node[below,yshift=-5pt] {$11\rlap{${}=|\bm s|$}$};
\draw[green!70!black]
	(1,.8)
		coordinate (s1-A)
		-- ++(3,0)
		node[above,midway] {$s_1=4$}
		node[blue,above,midway,yshift=1.5em] {$m_1=2$}
		coordinate (s1-B)
	(5,.8)
		coordinate (s2-AB)
		node[above] {$s_2=1$}
		node[brown,above,yshift=1.5em,xshift=.5cm] {$u_1=2$}
	(6,.8)
		coordinate (s3-AB)
		node[above] {$s_3=1$}
	(7,.8)
		coordinate (s4-A)
		-- ++(2,0)
		node[above,midway] {$s_4=3$}
		node[blue,above,midway,yshift=1.5em] {$m_2=1$}
		coordinate (s4-B)
	(10,.8)
		coordinate (s5-A)
		-- ++(1,0)
		coordinate (s5-B)
		node[above,midway] {$s_5=2$}
		node[blue,above,midway,yshift=1.5em] {$m_3=0$}
	(9.3,.8)
		coordinate (s6-AB)
		node[brown,above,yshift=1.5em] {$u_2=0$}
	;
\foreach\n in {1,4,5,6,7,9,10,11}
	\draw[help lines,densely dotted,thick] (\n,0) -- ++(0,.8);
\fill[green!70!black]
	(s1-A) circle (1.5pt)
	(s1-B) circle (1.5pt)
	(s2-AB) circle (1.5pt)
	(s3-AB) circle (1.5pt)
	(s4-A) circle (1.5pt)
	(s4-B) circle (1.5pt)
	(s5-A) circle (1.5pt)
	(s5-B) circle (1.5pt)
	;
\draw[densely dotted,help lines,thick]
	($(formula.177)+(0,.2)$) coordinate (A) -- ($(s1-A)-(0,1.5)$) coordinate (XA)
	($(A)+(2.9,0)$) coordinate (B) -- ($(s1-B)-(0,1.5)$) coordinate (XB)
	($(A)+(3.2,0)$) coordinate (C) -- ($(s4-A)-(0,1.5)$) coordinate (XC)
	($(A)+(6,0)$) coordinate (D) -- ($(s4-B)-(0,1.5)$) coordinate (XD)
	($(A)+(6.3,0)$) coordinate (E) -- ($(s5-A)-(0,1.5)$) coordinate (XE)
	($(A)+(9,0)$) coordinate (F) -- ($(s5-B)-(0,1.5)$) coordinate (XF);
\fill[gray!10]
 	(A) -- (XA) -- (XB) -- (B)
 	(C) -- (XC) -- (XD) -- (D)
 	(E) -- (XE) -- (XF) -- (F);
\draw[draw=gray]
	($(current bounding box.north west)+(-.3,.3)$)
		rectangle
	($(current bounding box.south east)+(.3,-.3)$);
\end{tikzpicture}
\end{center}
\caption{The case $\bm s=(4,1,1,3,2)$}
\label{fg:example}
\end{figure}

By the above facts and by Theorem~\ref{th:main}, we therefore obtain
\begin{align}
\notag
\lim_{n\to\infty}M_n^{(\bm s)}(a,p)
 &=\sum_{n_1\ge\cdots\ge n_{|\bm s|}\ge 1}
	 \frac{%
	 	\prod_{r=1}^{\|\bm s\|}
	 	(1-p)^{n_{|\bm s|_{r-1}+1}-n_{|\bm s|_r}}}%
	 {n_1\cdots n_{|\bm s|}}\cdot
	 \Bigl[
		(1-p+ap)^{n_{|\bm s|}}
	 -(1-p)^{n_{|\bm s|}}
	 \Bigr]
	 \\[1em]
\notag
 &=\sum_{n_1\ge\cdots\ge n_{|\bm s|}\ge 1}
	 \frac{%
	 	\prod_{r=1}^d
		\left[
			(1-p)^{n_{|\bm m|_{r-1}+|\bm u|_{r-1}+2r-1}}\cdot\bigl(\frac1{1-p}\bigr)^{n_{|\bm m|_r+|\bm u|_{r-1}+2r}}
		\right]
			}{n_1\cdots n_{|\bm s|}}\times\\
\notag
 &{}\hskip3cm
	\times
	 \Bigl[
		(1-p+ap)^{n_{|\bm s|}}
	 -(1-p)^{n_{|\bm s|}}
	 \Bigr]
	 \\[1em]
\notag
 &=\sum_{n_1\ge\cdots\ge n_{|\bm s|}\ge 1}
	 \frac{%
	 	\prod_{r=1}^{d-1}
		\left[
			(1-p)^{n_{|\bm m|_{r-1}+|\bm u|_{r-1}+2r-1}}\cdot\bigl(\frac1{1-p}\bigr)^{n_{|\bm m|_r+|\bm u|_{r-1}+2r}}
		\right]
		}{n_1\cdots n_{|\bm s|}}\times
	\\
\notag
 &{}\hskip3cm\times(1-p)^{n_{|\bm m|_{d-1}+|\bm u|_{d-1}+2d-1}}\cdot\bigl(\tfrac1{1-p}\bigr)^{n_{|\bm m|_d+|\bm u|_{d-1}+2d}}\times\\
\notag
 &{}\hskip3cm
	\times
 \Bigl[
	(1-p+ap)^{n_{|\bm s|}}
 -(1-p)^{n_{|\bm s|}}
 \Bigr]
	\\[1em]
\notag
 &=\sum_{n_1\ge\cdots\ge n_{|\bm s|}\ge 1}
	 \frac{%
	 	\prod_{r=1}^{d-1}
		\left[
			(1-p)^{n_{|\bm m|_{r-1}+|\bm u|_{r-1}+2r-1}}\cdot\bigl(\frac1{1-p}\bigr)^{n_{|\bm m|_r+|\bm u|_{r-1}+2r}}
		\right]
		}{n_1\cdots n_{|\bm s|}}\times
	\\
\notag
 &{}\hskip3cm\times(1-p)^{n_{|\bm m|_{d-1}+|\bm u|_{d-1}+2d-1}}\cdot
 	\Bigl[
	\bigl(1+\tfrac{ap}{1-p}\bigr)^{n_{|\bm s|}}
 -1\Bigr]
	\\[1em]
\notag
 &=\Li^\star_{\{1\}_{|\bm s|}}
		\left(
			\textstyle
			\bigsqcup_{i=1}^{d-1}
			\bigl\{
				1-p,\{1\}_{m_i},\frac1{1-p},\{1\}_{u_i}
			\bigr\},
			1-p,\{1\}_{m_d},1+\tfrac{ap}{1-p}
		\right)
	\\[.3em]
\label{eq:diff.polylog}
 &-\Li^\star_{\{1\}_{|\bm s|}}
		\left(
			\textstyle
			\bigsqcup_{i=1}^{d-1}
			\bigl\{
				1-p,\{1\}_{m_i},\frac1{1-p},\{1\}_{u_i}
			\bigr\},
			1-p,\{1\}_{m_d+1}
		\right),
\end{align}
where we assumed the convergence of both multiple polylogarithms and used the fact that $|\bm m|_d+|\bm u|_{d-1}+2d=|\bm m|+|\bm u|+2d=|\bm s|$.

Based on the obtained form of $M_n^{(\bm s)}(a,p)$ for $n\to\infty$ in \eqref{eq:diff.polylog}, we see that for proving Theorem~\ref{th:Li.1}, it suffices to show that
\begin{equation}
\label{eq:zero}
\lim_{n\to\infty}
	\Bigl(
		\Li^\star_{\bm s}
		\left(
			\{1\}_{\|\bm s\|-1},a
		\right)
	 -M_n^{(\bm s)}(a,p)
	\Bigr)=0
\end{equation}
for suitable values of the involved parameters $a,p$. We employ the Toeplitz theorem to prove~\eqref{eq:zero}. Temporarily, we consider the involved parameters $a,p$ as free, and the conditions put on them will be deduced during the rest of this proof.

Clearly, for $a=0$, Theorem~\ref{th:Li.1} is true. Therefore, we consider the case $a\neq 0$ from now on. In addition to this, to ensure the convergence of $\Li^\star_{\bm s}\left(\{1\}_{\|\bm s\|-1},a\right)$, it is necessary to assume that~$|a|\le 1$. Recall that $a=1$ is admissible with $s_1>1$ only, which is satisfied concerning the form of $\bm s$ in this proof. Further, we proceed as follows:
\begin{align*}
\Li^\star_{\bm s}
	\left(
	 \{1\}_{\|\bm s\|-1},a
 \right)
 &-M_n^{(\bm s)}(a,p)\\[.5em]
 &=\sum_{k=0}^n\binom nk\cdot p^k\cdot (1-p)^{n-k}\cdot
	 \Bigl(
		\Li^\star_{\bm s}\left(\{1\}_{\|\bm s\|-1},a\right)
	 -\zeta_k^\star(\bm s;a)
	 \Bigr)\\[.5em]
 &=\sum_{k=0}^n\binom nk\cdot p^k\cdot (1-p)^{n-k}\cdot
	 \sum_{n_1\ge\cdots\ge n_{\|\bm s\|}>k}
	 \frac{a^{n_{\|\bm s\|}}}{\prod_{i=1}^{\|\bm s\|}n_i^{s_i}}\\[.5em]
 &=\sum_{k=0}^n\binom nk\cdot (ap)^k\cdot (1-p)^{n-k}\cdot
	 \sum_{n_1\ge\cdots\ge n_{\|\bm s\|}\ge 1}
	 \frac{a^{n_{\|\bm s\|}}}{\prod_{i=1}^{\|\bm s\|}(n_i+k)^{s_i}}.
\end{align*}
Now, set for every $k,n\in\N$
\begin{align*}
a_{nk}
	&=\binom nk\cdot (ap)^k\cdot (1-p)^{n-k},\\[.5em]
x_k
 &=\sum_{n_1\ge\cdots\ge n_{\|\bm s\|}\ge 1}
	 \frac{a^{n_{\|\bm s\|}}}{\prod_{i=1}^{\|\bm s\|}(n_i+k)^{s_i}}.
\end{align*}

For the proof of \eqref{eq:zero}, it suffices to show that $\sum_{k=0}^na_{nk}\cdot x_k\to 0$ for our choice of $a_{nk}$ and~$x_k$ which can be achieved by the Toeplitz theorem. Let us verify its conditions. To fulfill condition~(a), it is necessary that for fixed $k\in\mathbb N$ we have $\lim_{n\to\infty}a_{nk}=0$ which is met whenever $|1-p|<1$. To fulfill condition (b), the sum $\sum_{k=0}^n|a_{nk}|$ must be bounded. Since
$
\sum_{k=0}^n|a_{nk}|
	=\sum_{k=0}^n\binom nk\cdot\left|(ap)^k\cdot(1-p)^{n-k}\right|
	=\left(|ap|+|1-p|\right)^n,
$
we set $|ap|+|1-p|\le 1$. Consequently, the intersection of the conditions $|1-p|<1$ and $|ap|+|1-p|\le 1$ implies $|a|\le\min\bigl(1,2/p-1\bigr)$ concluding the proof of Theorem~\ref{th:Li.1}.
\qed
\end{proof}

Theorem~\ref{th:Li.1} implies the following identity relating arbitrary multiple zeta-star value $\zeta^\star(\bm s)$ to a parametric difference of specific multiple polylogarithms.

\begin{corollary}
\label{cr:Li.1}
Assume that $m_i,u_i\in\N_0$, $d\in\N$, and
$
\bm s
	=\left(
		\textstyle\sqcup_{i=1}^{d-1}\{m_i+2,\{1\}_{u_i}\},m_d+2
	 \right).
$
If $p\in (0,1)$ then
\begin{align}
\notag
\zeta^\star(\bm s)
 &=\Li^\star_{\{1\}_{|\bm s|}}
		\left(
			\textstyle
			\bigsqcup_{i=1}^{d-1}
			\bigl\{
				1-p,\{1\}_{m_i},\frac1{1-p},\{1\}_{u_i}
			\bigr\},
			1-p,\{1\}_{m_d},\tfrac1{1-p}
		\right)\\[.5em]
\label{eq:Li.1}
 &\,-\Li^\star_{\{1\}_{|\bm s|}}
		\left(
			\textstyle
			\bigsqcup_{i=1}^{d-1}
			\bigl\{
				1-p,\{1\}_{m_i},\frac1{1-p},\{1\}_{u_i}
			\bigr\},
			1-p,\{1\}_{m_d+1}
		\right).
\end{align}
\end{corollary}

\begin{proof}
The statement follows immediately from setting $a=1$ in Theorem~\ref{th:Li.1}. The condition $p\in(0,1)$ can be deduced from solving the condition $|a|\le\min(1,2/p-1)$ with $a=1$.
\qed
\end{proof}

Setting $m_i=u_i=0$ for all $i=1,\dots,d$ in Corollary~\ref{cr:Li.1}, we easily deduce the following relation.

\begin{example}
\label{ex:Li.1}
Assume that $p\in (0,1)$. Then
\[
\zeta^\star(\{2\}_d)
	=\Li^\star_{\{1\}_{2d}}
		\left(
			\textstyle
			\bigl\{
				1-p,\frac1{1-p}
			\bigr\}_d
		\right)\\[.5em]
  -\Li^\star_{\{1\}_{2d}}
		\left(
			\textstyle
			\bigl\{
				1-p,\frac1{1-p}
			\bigr\}_{d-1},
			1-p,1
		\right).
\]
\end{example}

\begin{corollary}
\label{cr:Li.11}
Assume that $m_i,u_i\in\N_0$, $d\in\N$, and
$
\bm s
	=\left(
		\textstyle\sqcup_{i=1}^{d-1}\{m_i+2,\{1\}_{u_i}\},m_d+2
	 \right).
$
If $(a,p)\in[-1,1/3]\times[ 1/2,3/2]$ with $p\neq 1$ and $|a|\le\min(1,2/p-1)$ then
\begin{align}
\label{eq:Li.111}
&\Li^\star_{\{1\}_{|\bm s|}}
		\left(
			\textstyle
			\bigsqcup_{i=1}^{d-1}
			\bigl\{
				1-p,\{1\}_{m_i},\frac1{1-p},\{1\}_{u_i}
			\bigr\},
			1-p,\{1\}_{m_d+1}
		\right)
  =-\Li^\star_{\bm s}
		\left(
			\{1\}_{\|\bm s\|-1},1-\tfrac1p
		\right),\\[.5em]
\label{eq:Li.112}
&\Li^\star_{\{1\}_{|\bm s|}}
		\left(
			\textstyle
			\bigsqcup_{i=1}^{d-1}
			\bigl\{
				1-p,\{1\}_{m_i},\frac1{1-p},\{1\}_{u_i}
			\bigr\},
			1-p,\{1\}_{m_d},
			1+\frac{ap}{1-p}
		\right)\\[.5em]
\notag
 &\hskip5cm
 	=\Li^\star_{\bm s}
	 \left(
		\{1\}_{\|\bm s\|-1},a
	 \right)
	-\Li^\star_{\bm s}
	 \left(
		\{1\}_{\|\bm s\|-1},1-\tfrac1p
	 \right).
\end{align}
\end{corollary}

The importance of this statement consists in the essential reduction of computational complexity of the multiple polylogarithms in \eqref{eq:Li.111} and~\eqref{eq:Li.112}. While the depth of the polylogarithms on the left-hand sides is $|\bm s|$, the polylogarithms on the right-hand sides are of depth~$\|\bm s\|$, which is always less than $|\bm s|$ for arbitrary~$\bm s$ considered in this section. For instance, with $\bm s=(23,4)$ we have $|\bm s|=27$, whereas $\|\bm s\|=2$. Hence, the reduction effect is essential in certain instances.

\begin{proof}[of Corollary~\ref{cr:Li.11}]
For the proof of the first identity, set $a=1-1/p$ in~\eqref{eq:Li.10} and remember that
\begin{equation}
\label{eq:result}
\Li^\star_{\{1\}_{|\bm s|}}
		\left(
			\textstyle
			\bigsqcup_{i=1}^{d-1}
			\bigl\{
				1-p,\{1\}_{m_i},\frac1{1-p},\{1\}_{u_i}
			\bigr\},
			1-p,\{1\}_{m_d},0
		\right)=0
\end{equation}
due to the last zero argument. Now, \eqref{eq:Li.111} follows by applying~\eqref{eq:result} to~\eqref{eq:Li.10} with $a=1-1/p$. Of course, the above manipulations make sense only for $a,p\in\R$ with $|a|\le\min (1,2/p-1)$ satisfying $a=1-1/p$. Solving this system of conditions with respect to $p$ implies $p\in[1/2,3/2]$.

For the proof of \eqref{eq:Li.112}, apply identity~\eqref{eq:Li.111} with $p\in[ 1/2,3/2]$ to~\eqref{eq:Li.10} valid for $a,p\in\R$ with $|a|\le\min (1,2/p-1)$. Solving the intersection of these two condition groups implies $a\in[-1,1/3]$, and the proof follows.
\qed
\end{proof}

\subsection{Second theorem on polylogarithms}
\label{ssec:polylog.2}

This section is devoted to our second theorem on polylogarithms relating to the choice
$\bm s
=\left(
	\textstyle\sqcup_{i=1}^d\{m_i+2,\{1\}_{u_i}\}
 \right)$, $m_i,u_i\in\N_0$ with $u_d\ge 1$.
Even if the form of both tuples $\bm s$ in Sect.~\ref{ssec:polylog.1} and \ref{ssec:polylog.2} are of the same form up to the ending group of ones, the corresponding results cannot be put into one theorem easily.

\begin{theorem}
\label{th:Li.2}
Assume that $m_i,u_i\in\N_0$, $d,u_d\in\N$, and
$
\bm s
	=\left(
		\sqcup_{i=1}^d\{m_i+2,\{1\}_{u_i}\}
	 \right)
$. If $p\neq 1$ and $|a|\le \min(1,2/p-1)$ then
\begin{align}
\notag
\Li^\star_{\bm s}
 	\left(
		\{1\}_{\|\bm s\|-1},a
	\right)
\notag
 &=\Li^\star_{\{1\}_{|\bm s|}}
		\left(
			\textstyle
			\bigsqcup_{i=1}^d
			\bigl\{
				1-p,\{1\}_{m_i},\frac1{1-p},\{1\}_{u_i-\delta_{i,d}}
			\bigr\},
			1-p+ap
		\right)
	\\[.5em]
\label{eq:Li.2}
 &{}\,-\Li^\star_{\{1\}_{|\bm s|}}
		\left(
			\textstyle
			\bigsqcup_{i=1}^{d-1}
			\bigl\{
				1-p,\{1\}_{m_i},\frac1{1-p},\{1\}_{u_i-\delta_{i,d}}
			\bigr\},
			1-p
		\right),
\end{align}
where $\delta_{i,j}$ denotes the Kronecker delta.
\end{theorem}

\begin{remark}.
If we experimentally allow $u_d=0$ in Theorem~\ref{th:Li.2}, i.e.~$\bm s$ does not end with a group of ones, then the right-hand side of \eqref{eq:Li.2} is meaningless due to $\{1\}_{-1}$ originating from the argument~$\{1\}_{u_d-1}$. Of course, the `experimental' case $u_d=0$ in Theorem~\ref{th:Li.2} is handled by Theorem~\ref{th:Li.1}.
\end{remark}

We omit the proof of Theorem~\ref{th:Li.2}, for the method of deducing this statement is identical to that in the proof of Theorem~\ref{th:Li.1}. We now provide further particular cases following from Theorem~\ref{th:Li.2}.

\begin{corollary}
\label{cr:Li.21}
Assume that $m_i,u_i\in\N_0$, $d,u_d\in\N$, and
$
\bm s
	=\left(
		\sqcup_{i=1}^d\{m_i+2,\{1\}_{u_i}\}
	 \right)
$. If $p\in(0,1)$ then
\begin{align*}
\zeta^\star({\bm s})
 &=\Li^\star_{\{1\}_{|\bm s|}}
		\left(
			\textstyle
			\bigsqcup_{i=1}^d
			\bigl\{
				1-p,\{1\}_{m_i},\frac1{1-p},\{1\}_{u_i}
			\bigr\}
		\right)
	\\[.5em]
 &\,-\Li^\star_{\{1\}_{|\bm s|}}
		\left(
			\textstyle
			\bigsqcup_{i=1}^d
			\bigl\{
				1-p,\{1\}_{m_i},\frac1{1-p},\{1\}_{u_i-\delta_{i,d}}
			\bigr\},
			1-p
		\right),
\end{align*}
where $\delta_{i,j}$ denotes the Kronecker delta.
\end{corollary}

\begin{proof}
Set $a=1$ in Theorem~\ref{th:Li.2}.
\qed
\end{proof}

In a similar vein as we utilized Corollary~\ref{cr:Li.1} for deducing Example~\ref{ex:Li.1}, we obtain an analogous evaluation following from Corollary~\ref{cr:Li.21}.

\begin{example}
\label{ex:Li.2}
Assume that $d\in\N$ and $p\in (0,1)$. Then
\[
 \Li^\star_{\{1\}_{2d+1}}
		\left(
			\textstyle
			\bigl\{
				1-p,\frac1{1-p}
			\bigr\}_d,1
		\right)
 -\Li^\star_{\{1\}_{2d+1}}
		\left(
			\textstyle
			\bigl\{
				1-p,\frac1{1-p}
			\bigr\}_d,
			1-p
		\right)
 =2\,\zeta(2d+1).
\]
\end{example}

\begin{proof}
Setting $m_i=d_i=0$, $i=1,\dots,d-1$, and $m_d=0$, $u_d=1$, we infer by Corollary~\ref{cr:Li.21} that
\[
\zeta^\star(\{2\}_d,1)
	=\Li^\star_{\{1\}_{2d+1}}
		\left(
			\textstyle
			\bigsqcup_{i=1}^d
			\bigl\{
				1-p,\frac1{1-p}
			\bigr\},1
		\right)
	-\Li^\star_{\{1\}_{2d+1}}
		\left(
			\textstyle
			\bigsqcup_{i=1}^d
			\bigl\{
				1-p,\frac1{1-p}
			\bigr\},
			1-p
		\right).
\]
Using the evaluation $\zeta^\star(\{2\}_d,1)=2\,\zeta(2d+1)$, see e.g.~Genčev~\cite[Eq.~(42)]{gencev.mhfm}, we arrive at the identity in Example~\ref{ex:Li.2}.
\qed
\end{proof}

Finally, we append the following transformations of multiple polylogarithms.

\begin{corollary}
\label{cr:Li.22}
Assume that $m_i,u_i\in\N_0$, $d,u_d\in\N$, $p\in(0,1)$, and
$
\bm s
	=\left(
		\sqcup_{i=1}^d\{m_i+2,\{1\}_{u_i}\}
	 \right)
$. If $(a,p)\in[-1,1/3]\times[ 1/2,3/2]$ with $p\neq 1$ and $|a|\le\min(1,2/p-1)$ then
\begin{align*}
&\Li^\star_{\{1\}_{|\bm s|}}
	\left(
		\textstyle
		\bigsqcup_{i=1}^{d-1}
		\bigl\{
			1-p,\{1\}_{m_i},\frac1{1-p},\{1\}_{u_i-\delta_{i,d}}
		\bigr\},
		1-p
	\right)
 =-\Li^\star_{\bm s}
 		\left(
			\{1\}_{\|\bm s\|-1},1-\tfrac1p
		\right),\\[.5em]
&\Li^\star_{\{1\}_{|\bm s|}}
	\left(
		\textstyle
		\bigsqcup_{i=1}^d
		\bigl\{
			1-p,\{1\}_{m_i},\frac1{1-p},\{1\}_{u_i-\delta_{i,d}}
		\bigr\},
		1-p+ap
	\right)
	\\[.5em]
 &\hskip5cm
  =\Li^\star_{\bm s}
 		\left(
			\{1\}_{\|\bm s\|-1},a
		\right)
	-\Li^\star_{\bm s}
 		\left(
			\{1\}_{\|\bm s\|-1},1-\tfrac1p
		\right).
\end{align*}
\end{corollary}

We omit the proof because of the essential similarity to the proof of Corollary~\ref{cr:Li.11}.

\section{Arithmetic means of multiple harmonic-star numbers}
\label{sec:arithmetic}

With the help of the findings established in the previous sections, we develop additional results concerning arithmetic means of the generalized multiple harmonic-star numbers $\zeta^\star_k(\bm s;a)$, $k=0,\dots,n$, including the limiting case $n\to\infty$. Our main result in this section is the following general statement, which forms the basis for further applications.

\begin{theorem}
\label{th:mean}
Assume that $a\in\R$, $\bm s=(s_1,\dots,s_d)\in\N^d$, and define
\begin{equation}
\label{eq:Qs.def}
Q(\bm s)
	:=\sum_{r=1}^d
		\left(
			n_{|\bm s|_{r-1}+1}
		 -n_{|\bm s|_r}
		\right).
\end{equation}
Then
\begin{equation}
\label{eq:mean}
\frac1{n+1}\cdot
\sum_{k=1}^n\zeta^\star_k(\bm s;a)
	=\sum_{n\ge n_1\ge\cdots\ge n_{|\bm s|+1}\ge 1}
		\frac{%
			\binom{n_{|\bm s|}}{n_{|\bm s|+1}}}{%
			\binom{Q(\bm s)+n_{|\bm s|}}{n_{|\bm s|+1}}}
		\cdot
		\frac{
			a^{n_{|\bm s|+1}}}{%
			(Q(\bm s)+n_{|\bm s|}+1)\cdot n_1\cdots n_{|\bm s|}}.
\end{equation}
\end{theorem}

Even if the theorem is rather complex, we show that its particular cases are interesting as the limiting forms of \eqref{eq:mean} for $n\to\infty$ imply new identities. Before we approach these applications, we first prove Theorem~\ref{th:mean}.

\begin{proof}
The idea of the proof is straightforward; we perform the integration $\int_0^1M_n^{(\bm s)}(a,p)\,\mathrm dp$ in two different ways. The first integration consists in considering the defining expression of the sum~$M_n^{(\bm s)}(a,p)$ in \eqref{eq:xmneimneh.ap.def}, the second one in considering the transformed variant from Theorem~\ref{th:main}. Let us realize these steps below.

Integrating the defining expression given in~\eqref{eq:xmneimneh.ap.def} yields
\begin{align}
\notag
\int_0^1 M_n^{(\bm s)}(a,p)\,\mathrm dp
	&=\sum_{k=1}^n\binom nk\cdot\zeta_k^\star(\bm s;a)\cdot
		\int_0^1 p^k\cdot(1-p)^{n-k}\,\mathrm dp\\[.5em]
\label{eq:AP.1}
	&=\frac1{n+1}\cdot\sum_{k=1}^n\zeta_k^\star(\bm s;a),
\end{align}
where we used the beta-type integral evaluation $\int_0^1 p^k\cdot(1-p)^{n-k}\,\mathrm dp=1/\bigl((n+1)\cdot\binom nk\bigr)$.

In the second step, we employ the formula from Theorem~\ref{th:main}. Considering the definition of the $Q(\bm s)$ in \eqref{eq:Qs.def}, we obtain
\begin{align}
\notag
\int_0^1 &M_n^{(\bm s)}(a,p)\,\mathrm dp\\
	&=\int_0^1\sum_{n\ge n_1\ge\cdots\ge n_{|\bm s|}\ge 1}
		\frac{(1-p)^{Q(\bm s)}}{n_1\cdots n_{|\bm s|}}\cdot
		\Bigl[
		 (1-p+ap)^{n_{|\bm s|}}
		-(1-p)^{n_{|\bm s|}}
		\Bigr]\mathrm dp\\[.5em]
\notag
	&=\sum_{n\ge n_1\ge\cdots\ge n_{|\bm s|}\ge 1}\int_0^1
		\frac{%
			(1-p)^{Q(\bm s)}\cdot(1-p+ap)^{n_{|\bm s|}}
		 -(1-p)^{Q(\bm s)+n_{|\bm s|}}}{%
		 n_1\cdots n_{|\bm s|}}\,\mathrm dp\\[.5em]
\notag
	&=\sum_{n\ge n_1\ge\cdots\ge n_{|\bm s|}\ge 1}\int_0^1
		\frac{%
			\sum_{j=0}^{n_{|\bm s|}}
				\binom{n_{|\bm s|}}j\cdot
				(1-p)^{Q(\bm s)+j}\cdot (ap)^{n_{|\bm s|}-j}
		 -(1-p)^{Q(\bm s)+n_{|\bm s|}}}{%
		 n_1\cdots n_{|\bm s|}}\,\mathrm dp\\[.5em]
\notag
	&=\sum_{n\ge n_1\ge\cdots\ge n_{|\bm s|}\ge 1}
		\frac{%
			\sum_{j=0}^{n_{|\bm s|}}
				\binom{n_{|\bm s|}}j\cdot a^{n_{|\bm s|}-j}\cdot
				\frac1{(Q(\bm s)+n_{|\bm s|}+1)\cdot\binom{Q(\bm s)+n_{|\bm s|}}{n_{|\bm s|}-j}}
			 -\frac1{Q(\bm s)+n_{|\bm s|}+1}}{%
		 n_1\cdots n_{|\bm s|}}\\[.5em]
\notag
	&=\sum_{n\ge n_1\ge\cdots\ge n_{|\bm s|}\ge 1}
		\frac{%
			\sum_{j=0}^{n_{|\bm s|}}
				\frac{\binom{n_{|\bm s|}}j}{\binom{Q(\bm s)+n_{|\bm s|}}{n_{|\bm s|}-j}}\cdot a^{n_{|\bm s|}-j}
			 -1}{%
			 (Q(\bm s)+n_{|\bm s|}+1)\cdot
			 n_1\cdots n_{|\bm s|}}\\[.5em]
\label{eq:AP.2}
	&=\sum_{n\ge n_1\ge\cdots\ge n_{|\bm s|+1}\ge 1}
		\frac{%
			\binom{n_{|\bm s|}}{n_{|\bm s|+1}}}{%
			\binom{Q(\bm s)+n_{|\bm s|}}{n_{|\bm s|+1}}}
		\cdot
		\frac{%
			a^{n_{|\bm s|+1}}}{%
			(Q(\bm s)+n_{|\bm s|}+1)\cdot
			n_1\cdots n_{|\bm s|}},
\end{align}
where we used thy symmetry of binomial coefficients $\binom nj=\binom n{n-j}$ in the last step. Comparing the expressions in \eqref{eq:AP.1} and~\eqref{eq:AP.2} concludes the proof.
\qed
\end{proof}

Recall that the expression $Q(\bm s)$ can essentially be simplified if the components of $\bm s$ are equal to $1$. We thus present the form of Theorem~\ref{th:mean} for $\bm s=(\{1\}_d)$ in the following example.

\begin{example}
\label{ex:mean.1}
Assume that $d,n\in\N$. Then
\[
\frac1{n+1}\cdot\sum_{k=1}^n\zeta^\star_k(\{1\}_d)
	=\sum_{n\ge n_1\ge\cdots\ge n_d\ge 1}\frac1{n_1\cdots n_{d-1}\cdot(n_d+1)},
\]
where the product $n_1\cdots n_{d-1}$ is considered as empty for $d=1$.
\end{example}
\goodbreak

Notice that for $d=1$, Example~\ref{ex:mean.1} implies the identity $\sum_{k=1}^n H_k=(n+1)\cdot(H_{n+1}-1)$, which is well known.

\begin{proof}[of Example~\ref{ex:mean.1}]
Setting $\bm s=(\{1\}_d)$ in Theorem~\ref{th:mean} implies $Q(\{1\}_d)=0$ following from the definition of $Q(\bm s)$ in~\eqref{eq:Qs.def}. Therefore, from $|\bm s|=d$ and by \eqref{eq:mean}, we obtain
\begin{align*}
\frac1{n+1}\cdot\sum_{k=1}^n\zeta^\star_k(\{1\}_d)
 &=\sum_{n\ge n_1\ge\cdots\ge n_{d+1}\ge 1}
		\frac{%
			\binom{n_d}{n_{d+1}}}{%
			\binom{n_d}{n_{d+1}}}\cdot
		\frac1{n_1\cdots n_d\cdot(n_d+1)}\\[.5em]
 &=\sum_{n\ge n_1\ge\cdots\ge n_d\ge 1}
		\frac1{n_1\cdots n_d\cdot(n_d+1)}
	 \sum_{n_{d+1}=1}^{n_d}1\\[.5em]
 &=\sum_{n\ge n_1\ge\cdots\ge n_d\ge 1}
		\frac1{n_1\cdots n_{d-1}\cdot(n_d+1)},
\end{align*}
which concludes the proof.
\qed
\end{proof}

Let us now study particular cases of Theorem~\ref{th:mean} for $n\to\infty$. We present two corollaries offering new series representations of $\Li^\star_{\bm s}(\{1\}_{\|\bm s\|-1},a)$ and $\zeta^\star(\bm s)$.

\begin{corollary}
\label{cr:mean.1}
Assume that $a\in\R$, $|a|\le 1$, $\bm s=(s_1,\dots,s_d)\in\N^d$, and let $Q(\bm s)$ be defined in~\eqref{eq:Qs.def}. Then
\begin{equation}
\label{eq:mean.1}
\Li^\star_{\bm s}(\{1\}_{\|\bm s\|-1},a)
 =\sum_{n_1\ge\cdots\ge n_{|\bm s|+1}\ge 1}
	\frac{%
		\binom{n_{|\bm s|}}{n_{|\bm s|+1}}}{%
		\binom{Q(\bm s)+n_{|\bm s|}}{n_{|\bm s|+1}}}
	\cdot
	\frac{
		a^{n_{|\bm s|+1}}}{%
		(Q(\bm s)+n_{|\bm s|}+1)\cdot n_1\cdots n_{|\bm s|}}
\end{equation}
whenever the polylogarithm on the left-hand side converges.
\end{corollary}

\begin{proof}
The proof is straightforward; it suffices to calculate the limits of both sides of \eqref{eq:mean} for~$n\to\infty$. Of course, the limit of the left-hand side of \eqref{eq:mean} can be calculated with the help of Remark~\ref{rm:toeplitz}. Consequently, we recognize the value of this limit as $\Li^\star_{\bm s}(\{1\}_{\|\bm s\|-1},a)$ assuming its convergence. Taking into account the simpler limit of the right-hand side of~\eqref{eq:mean} completes the proof since the existence of this limit is uniquely determined by the convergence of~$\Li^\star_{\bm s}(\{1\}_{\|\bm s\|-1},a)$ on the left, which was assumed above.
\qed
\end{proof}

\begin{corollary}
\label{cr:mean.2}
Assume that $\bm s=(s_1,\dots,s_d)\in\N^d$, $s_1>1$, and let $Q(\bm s)$ be defined in \eqref{eq:Qs.def}. Then
\begin{equation}
\label{eq:mean.2}
\zeta^\star(\bm s)
 =\sum_{n_1\ge\cdots\ge n_{|\bm s|}\ge 1}
	\frac1{%
		(Q(\bm s)+1)\cdot(Q(\bm s)+n_{|\bm s|}+1)\cdot n_1\cdots n_{|\bm s|-1}}.
\end{equation}
\end{corollary}

\begin{proof}
Setting $a=1$ into the right-hand side of \eqref{eq:mean.1}, we obtain $\Li^\star_{\bm s}(\{1\}_{\|\bm s\|-1},1)=\zeta^\star(\bm s)$ representing a convergent value due $s_1>1$. On the contrary, the right-hand side of~\eqref{eq:mean.1} reads for $a=1$
\begin{align*}
\sum_{n_1\ge\cdots\ge n_{|\bm s|+1}\ge 1}
 &\frac{%
		\binom{n_{|\bm s|}}{n_{|\bm s|+1}}}{%
		\binom{Q(\bm s)+n_{|\bm s|}}{n_{|\bm s|+1}}}
	\cdot
	\frac1{(Q(\bm s)+n_{|\bm s|}+1)\cdot n_1\cdots n_{|\bm s|}}
	\\[.5em]
 &=\sum_{n_1\ge\cdots\ge n_{|\bm s|}\ge 1}
		\frac1{(Q(\bm s)+n_{|\bm s|}+1)\cdot n_1\cdots n_{|\bm s|}}
		\cdot
		\sum_{n_{|\bm s|+1}=1}^{n_{|\bm s|}}
		\frac{%
			\binom{n_{|\bm s|}}{n_{|\bm s|+1}}}{%
			\binom{Q(\bm s)+n_{|\bm s|}}{n_{|\bm s|+1}}}\\
 &=\sum_{n_1\ge\cdots\ge n_{|\bm s|}\ge 1}
		\frac1{(Q(\bm s)+n_{|\bm s|}+1)\cdot n_1\cdots n_{|\bm s|}}
		\cdot\frac{n_{|\bm s|}}{Q(\bm s)+1},
\end{align*}
where we used the identity $\sum_{k=1}^m\binom mk/\binom nk=\frac m{n+1-m}$. After canceling the variable~$n_{|\bm s|}$, the proof is completed.
\qed
\end{proof}

\begin{example}
\label{ex:mean.2}
Assume that $d\in\N$. Then
\begin{equation}
\label{eqx:mean.2}
\sum_{n_1\ge\cdots\ge n_{2d}\ge 1}
	\frac1{%
		\left(1+\sum_{i=1}^{2d}(-1)^{i-1}\cdot n_i\right)\cdot\left(1+\sum_{i=1}^{2d-1}(-1)^{i-1}\cdot n_i\right)\cdot n_1\cdots n_{2d-1}}
	=\zeta^\star(\{2\}_d).
\end{equation}
\end{example}

It is worth mentioning that the last formula for $d=1$ is well known. In this case, the identity in \eqref{eqx:mean.2} reduces to
\[
\sum_{n_1\ge n_2\ge 1}\frac1{(1+n_1-n_2)\cdot(1+n_1)\cdot n_1}=\zeta(2)
\]
whose left-hand side can be written as $\sum_{n_1\ge 1}H_{n_1}/(n_1\cdot(1+n_1))$, where $H_{n_1}$ denotes the harmonic number.

\begin{proof}[of Example~\ref{ex:mean.2}]
It is not hard to see that for $\bm s=(\{2\}_d)$, the definition of $Q(\bm s)$ in~\eqref{eq:Qs.def} implies
$
Q(\bm s)
	=\sum_{i=1}^d(n_{2i-1}-n_{2i})
	=\sum_{i=1}^{2d}(-1)^{i-1}\cdot n_i.	
$
Applying this formula to \eqref{eq:mean.2} and performing simplifications, we immediately obtain~\eqref{eqx:mean.2}.
\qed
\end{proof}

\begin{remark}
We point out the well-known evaluation $\zeta^\star(\{2\}_d)=\bigl(2-4^{1-d}\bigr)\cdot\zeta(2d)$, that can be applied to~\eqref{eqx:mean.2}.
\end{remark}

\section*{Acknowledgement}

The author would like to thank Prof.~Ce Xu (Anhui Normal University) for informing me about the conjecture in~\cite{pan.xu} and for his support during the course of this work.

\section*{Conflict of interest}

The author declares that he has no conflict of interest.

\section*{Funding}

No funding was received for conducting this study.
The author has no relevant financial or non-financial interests to disclose.

\section*{Declaration of Generative AI and AI-assisted technologies in the writing process}

No content generated by AI technologies has been used in this work.

\section*{Data availability}

No data was used for the research described in the article.

\bibliographystyle{spmpsci}
\bibliography{database}

\end{document}